\documentclass[a4paper,12pt, oneside]{article}

\usepackage{amsmath, amssymb, amsthm, pict2e, centernot}
\usepackage{mathtools, faktor, enumitem}
\usepackage{tikz-cd} 

\setenumerate[1]{label={(\arabic*)}}

\usepackage[T1]{fontenc}
\usepackage{ae}
\usepackage{aecompl}
\usepackage{textcomp}
\usepackage[utf8]{inputenc}
\usepackage{lmodern}
\usepackage{microtype}

\newlength{\alphabet}
\usepackage{geometry}
\usepackage[backend=biber, style=numeric, maxbibnames=99, doi=false]{biblatex}
\AtEveryBibitem{\clearlist{language}\clearfield{note}}
\AtEveryBibitem{%
  \ifentrytype{misc}
    {\clearfield{url}\clearfield{urlyear}}
    {}%
}
\usepackage{hyperref}
\hypersetup{
  colorlinks=true,    
  urlcolor=gray,      
  linkcolor=darkgray, 
  citecolor=gray 
}

\DeclareFieldFormat*{title}{\mkbibemph{#1}}
\DeclareFieldFormat*{journaltitle}{#1}
\DeclareFieldFormat*{booktitle}{#1}

\renewbibmacro{in:}{%
  \ifentrytype{article}
    {}
    {\printtext{\bibstring{in}\intitlepunct}}}

\DeclareFieldFormat[article,periodical]{volume}{\mkbibbold{#1}}
\DeclareFieldFormat[article,periodical]{number}{\bibstring{number}~#1}
\renewbibmacro*{journal+issuetitle}{%
  \usebibmacro{journal}%
  \setunit*{\addspace}%
  \iffieldundef{series}
    {}
    {\newunit
     \printfield{series}%
     \setunit{\addspace}}%
  \printfield{volume}%
  \setunit{\addspace}%
  \usebibmacro{issue+date}%
  \setunit{\addcolon\space}%
  \usebibmacro{issue}%
  \setunit{\addcomma\space}%
  \printfield{number}%
  \setunit{\addcomma\space}%
  \printfield{eid}%
  \newunit}

\DeclareFieldFormat[article,periodical]{pages}{#1}

\DeclareFieldFormat{mrnumber}{%
  MR\addcolon\space
  \ifhyperref
    {\href{http://www.ams.org/mathscinet-getitem?mr=#1}{\nolinkurl{#1}}}
    {\nolinkurl{#1}}}

\renewbibmacro*{doi+eprint+url}{%
  \iftoggle{bbx:doi}
    {\printfield{doi}}
    {}%
  \newunit\newblock
  \printfield{mrnumber}%
  \newunit\newblock
  \iftoggle{bbx:eprint}
    {\usebibmacro{eprint}}
    {}%
  \newunit\newblock
  \iftoggle{bbx:url}
    {\usebibmacro{url+urldate}}
    {}}

\AtBeginBibliography{\small\sloppy}

\usepackage{hyperref}
\hypersetup{
    colorlinks=true,    
    urlcolor=gray,      
    linkcolor=darkgray, 
    citecolor=gray, 
    breaklinks=true,
    pdfusetitle=true
  }
\usepackage{breakurl}

\newcommand*{\Z}{\mathbb{Z}}
\newcommand*{\N}{\mathbb{N}}

\newcommand*{\C}{\mathbb{C}}

\newcommand*{\cO}{\mathcal{O}}
\newcommand*{\cT}{\mathcal{T}}

\newcommand*{\cI}{\mathcal{I}}
\newcommand*{\cK}{\mathcal{K}}
\newcommand*{\cL}{\mathcal{L}}
\newcommand*{\cB}{\mathcal{B}}

\newcommand*{\fM}{\mathfrak{M}}
\newcommand*{\fN}{\mathfrak{N}}

\newcommand*{\Ss}[1]{{#1}^{**}}
\newcommand*{\ssA}{\Ss{A}}
\newcommand*{\ssB}{\Ss{B}}
\newcommand*{\ssX}{\Ss{X}}

\makeatletter
\newcommand{\disj}{}
\DeclareRobustCommand{\disj}{%
  \mathrel{%
    \mathpalette\disj@\relaxn
  }%
}

\newcommand{\disj@}[2]{%
  \begingroup
  \setlength{\unitlength}{\disj@height{#1}}%
  \begin{picture}(1.4,2)
  \roundcap
  \linethickness{\disj@thickness{#1}}
  \put(0.7,0.5){\circle{1}}
  \Line(0.7,1)(0.7,2)
  \end{picture}%
  \endgroup
}
\newcommand{\disj@height}[1]{%
  0.8\fontdimen5
  \ifx#1\displaystyle\textfont\else
  \ifx#1\textstyle\textfont\else
  \ifx#1\scriptstyle\scriptfont\else
  \scriptscriptfont\fi\fi\fi 3
}
\newcommand{\disj@thickness}[1]{%
  \fontdimen8
  \ifx#1\displaystyle\textfont\else
  \ifx#1\textstyle\textfont\else
  \ifx#1\scriptstyle\scriptfont\else
  \scriptscriptfont\fi\fi\fi 3
}
\makeatother

\theoremstyle{plain}
\newtheorem{theorem}{Theorem}[section]
\newtheorem{lemma}[theorem]{Lemma}
\newtheorem{proposition}[theorem]{Proposition}
\newtheorem{corollary}[theorem]{Corollary}
\theoremstyle{definition}
\newtheorem{definition}[theorem]{Definition}
\newtheorem{example}[theorem]{Example}
\newtheorem{construction}[theorem]{Construction}

\newtheorem*{ack}{Acknowledgements}
\theoremstyle{remark}
\newtheorem*{remark}{Remark}

\title{Maximality of correspondence representations}
\author{Boris Bilich}
\date{}

\begin{document}
\begin{center}
	{\LARGE\bf
		Maximality of correspondence representations 
	}

	\bigskip

	Boris Bilich\footnote{University of Haifa and Georg-August-Universit\"{a}t G\"{o}ttingen \\
	The author is partially supported by the Bloom PhD scholarship and the DFG Middle-Eastern collaboration project no. 529300231\\


		}

\end{center}
\begin{abstract}
	In this paper, we fully characterize maximal representations of a C*-correspondence, thereby strengthening several earlier results.
	We demonstrate the maximality criteria through diverse examples.
	We also describe the noncommutative Choquet boundary and provide additional counterexamples to Arveson's hyperrigidity conjecture following the counterexample recently found by the author and Dor-On.
	Furthermore, we identify several classes of correspondences for which the hyperrigidity conjecture holds.
\end{abstract}
\section{Introduction}
A dilation of an operator $T$ on a Hilbert space $H$ is an operator $S$ on a larger Hilbert space $K\supset H$ such that $T$ is a corner of $S$, often represented in block-diagonal form as:
\[
	S = \begin{bmatrix}
		T    & \ast \\
		\ast & \ast
	\end{bmatrix}.
\]
The general paradigm of dilation theory is to study properties of operators by finding a dilation that exhibits features that are easier to work with.
This concept is best illustrated by classical results such as Sz.-Nagy's unitary dilation of a contraction \cite{sznagy}, And\^{o}'s dilation theorem \cite{ando}, which shows that a pair of commuting contractions can be dilated to a pair of commuting unitaries, and Stinespring's dilation \cite{stinespring} of a unital completely positive map on a C*-algebra to a *-homomorphism.
The theory is now flourishing \cite{dil-bhatt, choi-davidson, davidson-dor-on, shalit-product-sys, davidson-power, davidson-fuller, vernik, popescu} and has many applications. For a comprehensive introduction and overview, see \cite{shalit-guided-tour}.

This paper studies dilation theory for operator systems associated with C*-correspondences.
We first define what operator systems and dilations mean in this context to set the stage.

An \emph{operator system} $\mathcal{S}$ is a self-adjoint unital subspace of a C*-algebra $B$.
We additionally require that $\mathcal{S}$ generates $B$ as a C*-algebra.
A unital linear map $\varphi\colon \mathcal{S} \to \mathcal{B}(H)$ is called unital completely positive (u.c.p.~map) if it sends positive $n$ by $n$ matrices with entries in $\mathcal{S}$ to positive operators on $H^{\oplus n}$.
A dilation of $\varphi$ is a u.c.p.~map $\psi\colon \mathcal{S} \to \mathcal{B}(K)$ together with an isometric embedding $V\colon H \hookrightarrow K$ such that
\[
	V^* \psi(s) V = \varphi(s) \text{ for all } s \in \mathcal{S}.
\]
The dilation is called trivial if $\psi = \varphi \oplus \psi'$ under the identification of $H$ with $VH$.
A u.c.p.~map $\varphi$ is said to be maximal if it admits no non-trivial dilations.

Building on the ideas from Agler \cite{agler} and Muhly and Solel \cite{muhly-solel-boundary}, Dritschel and McCullough demonstrated that maximal u.c.p.~maps can be characterized by the \emph{unique extension property} (UEP) defined by Arveson~\cite{arveson69, arveson-note}.
By definition, a u.c.p.~map $\varphi\colon \mathcal{S} \to \mathcal{B}(H)$ satisfies the UEP if there is a *-representation $\rho\colon B \to \mathcal{B}(H)$ such that $\varphi = \rho|_{\mathcal{S}}$ and $\rho$ is the unique u.c.p.~extension of $\varphi$.
We say that a representation $\rho$ of $B$ is maximal (with respect to $\mathcal{S}$) if $\rho|_{\mathcal{S}}$ is a maximal u.c.p.~map.
By the above, there is a bijective correspondence between maximal u.c.p.~maps and maximal $B$-representations.

According to the theorem of Dritschel and McCullough~\cite[Theorem 2.1]{dritschel-mc}, every u.c.p.~map has a maximal dilation.
Dritschel and McCullough used maximality to provide the first dilation-theoretic proof of the existence of the C*-envelope $C^*_{env}(\mathcal{S})$, which is the smallest quotient of $B$ such that the quotient map is completely isometric on $\mathcal{S}$.
As they showed in \cite[Theorem 4.1]{dritschel-mc}, maximal representations factor through $C^*_{env}(\mathcal{S})$, and there is a maximal representation that is faithful on $C^*_{env}(\mathcal{S})$.

However, not every representation of $C^*_{env}(\mathcal{S})$ is maximal.
To better understand this phenomenon, Arveson \cite{arveson-choquet} defined the \emph{noncommutative Choquet boundary} $\partial_{C}\mathcal{S} \subset \widehat{C^*_{env}(\mathcal{S})}$ as the subset of unitary equivalence classes of maximal irreducible representations.
Elements of $\partial_C \mathcal{S}$ are called \emph{boundary representations}.
Arveson~\cite{arveson-choquet} in the separable case and Davidson-Kennedy~\cite{davidson-kennedy} in general proved that every operator system has sufficiently many boundary representations.

In \cite{arveson-hyperrigidity} Arveson introduced the notion of hyperrigidity of operator systems and gave several equivalent definitions.
The most relevant definition for us is that an operator system is hyperrigid if all representations of $C^*_{env}(\mathcal{S})$ are maximal.
In an attempt to connect this notion with the noncommutative Choquet boundary, Arveson conjectured \cite[Conjecture 4.3]{arveson-hyperrigidity} that if $\partial_C\mathcal{S} = \widehat{C^*_{env}(\mathcal{S})}$, i.e., all irreducible representations of the C*-envelope are boundary representations, then the operator system is hyperrigid.
This has come to be known as \emph{Arveson's hyperrigidity conjecture}, and it has received significant attention in the literature \cite{davidson-kennedy, kennedy-shalit, clouatre1,clouatre2, clouatre-hartz, clouatre-thompson, davidson-kennedy1, harris-kim, kim, katsoulis-ramsey, thompson, pietrzycki-stochel, kakariadis-shalit}.
Recently, an elementary Type I counterexample was found by the author and Dor-On \cite{hyp-counter}.
However, it is still interesting to determine for which classes of operator systems the conjecture holds.
For example, the conjecture remains open for function systems in commutative C*-algebras.

In the current paper, we study the topics described above in the context of a C*-correspondence $X$ over a C*-algebra $A$.
There are two natural operator systems associated with $X$ in the Cuntz-Toeplitz algebra $\cT_X$.
The first one is $\cT_X^+ + (\cT_X^+)^*$ for the non-self-adjoint tensor algebra $\cT_X^+$ generated by $A$ and $X$.
This operator algebra and the associated dilation theory were studied by Muhly and Solel \cite{muhly-solel-dilation} and by Katsoulis and Ramsey \cite{katsoulis-ramsey, katsoulis-ramsey2}.
Another operator system is $\mathcal{S}_X = X^* + A + X \subset \cT_X$ which was studied by Kim~\cite{kim}.
Moreover, following Muhly and Solel \cite{muhly-solel-dilation}, we give a definition of dilations and maximality for representations of $X$ without any reference to operator systems in Definition~\ref{d:dilation}.
We thus have, \emph{a priori}, three notions of maximality for isometric representations of $X$.
Fortunately, they all coincide (Proposition~\ref{p:uep-equiv}).

We study dilation theory for correspondences in two steps.
First, we analyze the case when $X$ is a  W*-correspondence over a von Neumann algebra $\fM$.
It turns out that the dilations in this case have a much more amenable structure.
Maximality is then detected by a spatial condition involving the \emph{support projection} $P_X\in \fM$ of $X$ as follows: a representation $t\colon X\to \cB(H)$ is maximal if and only if $\overline{t(X)H} = \rho(P_X)H$.
We call this property the \emph{full Cuntz-Pimsner covariance} since it generalizes the notion of full Cuntz-Krieger families by Dor-On and Salomon~\cite{doron-salomon} in the context of graph operator algebras.

After dealing with W*-correspondences, we reduce the problem for C*-correspondences to the W*-case by using second dual techniques (see Section~\ref{s:second-dual}).
If $X$ is a C*-correspondence over a C*-algebra $A$, then $\ssA$ is the universal von Neumann algebra of $A$, and $\ssX$ is a W*-correspondence over it (see \cite[Section 8.5]{BLM}).
We prove that the representations of $X$ are in canonical bijection with the normal representations of $\ssX$ and that this bijection respects dilations (Proposition~\ref{p:ss-equiv}).
Therefore, the notions of maximality for $X$ and $\ssX$ coincide and thus the maximality criterion directly follows from the one for W*-correspondences.
We define full Cuntz-Pimsner covariance for W*- and C*-correspondences simultaneously in Definition~\ref{d:p-relative} and prove that it is equivalent to maximality in Theorem~\ref{t:criterion}, which is the main result of the paper.
A similar idea to detect maximality by a projection in the second dual algebra was investigated by Clouâtre and Saikia \cite{clouatre-saikia}.
This shows that the reduction to von Neumann algebras via second duals is a powerful tool for studying non-commutative Choquet boundary theory.

Of course, the second dual $\ssA$ is a very big algebra which is hard to describe even for the most basic C*-algebras.
Therefore, we also prove several maximality criteria which do not refer to the second dual.
We describe maximality for an \emph{almost proper} correspondence in terms of ideals in Proposition~\ref{p:almost-proper}.
Furthermore, we interpret full Cuntz-Pimsner covariance for topological graphs in terms of receiving vertices in Theorem~\ref{t:topological-graphs}.

Using the maximality criterion, we are able to recover several earlier results.
The characterization of maximal representations of graph correspondences by Dor-On and Salomon \cite{doron-salomon} is an immediate consequence, and we describe it in Example~\ref{ex:dor-on-salomon} (and Theorem~\ref{t:topological-graphs} is a generalization to topological graphs).
In Corollary~\ref{cor:katsoulis-kribs}, we prove the Katsoulis-Kribs theorem, which states that the Cuntz-Pimsner algebra $\cO_X$ is the C*-envelope of $\cT^+_X$.

Kim~\cite{kim} established a general criterion for the hyperrigidity of a correspondence, refining the result of Katsoulis and Ramsey~\cite{katsoulis-ramsey}, who provided a sufficient condition for hyperrigidity when the correspondence is countably generated. In Corollary~\ref{cor:kim}, we derive Kim's hyperrigidity criterion: a correspondence $X$ is hyperrigid if and only if Katsura's ideal acts non-degenerately on $X$.

In Section~\ref{s:boundary}, we study the noncommutative Choquet boundary and Arveson's hyperrigidity conjecture.
We provide an abstract characterization of the Choquet boundary in Theorem~\ref{t:choquet} via the so-called atomic spectrum $\sigma_a(X) \subset \widehat{A}$ (see Definition~\ref{d:atomic}).
For Type I and abelian C*-algebras, we give a more concrete description of the atomic spectrum, and thus of the Choquet boundary, in Proposition~\ref{p:point-sp-typeI}.

We analyze the hyperrigidity conjecture for correspondences.
A whole class of counterexamples to the hyperrigidity conjecture is presented in Proposition~\ref{p:hyperrigidity-of-motimesa}.
We also derive a measure-theoretic criterion for a correspondence associated with a topological quiver to be a counterexample in Corollary~\ref{cor:conjecture-quiver}.
On the other hand, we show that the conjecture holds for proper correspondences (Proposition~\ref{p:conjecture-proper}) and for topological graphs (Corollary~\ref{cor:conjecture-quiver}).
This demonstrates that correspondences provide a natural framework for studying the connections between the Choquet boundary and hyperrigidity.

\begin{ack}
	The author is deeply grateful to his PhD advisor, Adam Dor-On, for formulating the original problem and providing guidance throughout this work. Special thanks are also due to Baruch Solel for a productive conversation about W*-correspondences. The author also thanks Orr Shalit for his insightful comments and corrections on the first version of the preprint.
\end{ack}

\section{C*- and W*-correspondences}\label{s:correspondences}
Throughout the section, $A$, $B$ will denote C*-algebras and $\fM$, $\fN$ will denote von Neumann algebras.

\begin{definition}
	An \emph{$A$-$B$-correspondence} is a right Hilbert $B$-module $X$ together with a non-degenerate left $A$-action $\varphi_X\colon A\to \cL(X)$.
	The correspondence $X$ is called \emph{proper} if $\varphi_X(A) \subset \cK(X)$ and \emph{injective} if $\varphi_X$ is injective.
	When $A=B$, we refer to it as an $A$-correspondence instead of an $A$-$A$-correspondence.
	We write $\varphi$ instead of $\varphi_X$ when $X$ is clear from the context.

	An $\fM$-$\fN$-correspondence $X$ is a W*-correspondence if $X$ has a Banach predual and the left action $\varphi_{X}$ is normal (see \cite[Section 8.5]{BLM} for equivalent definitions).
	Given von Neumann algebras $\fM$, $\fN$, and a correspondence $X$ between them, we will always assume that $X$ is a W*-correspondence.
\end{definition}
If $X$ is an $A$-$B$-correspondence and $Y$ is a $B$-$C$-correspondence, then there is an $A$-$C$-correspondence $X\otimes_{B} Y$ which is a completion of the algebraic tensor product with respect to the tensor inner product defined by
\[
	\langle \xi\otimes \eta, \xi'\otimes\eta'\rangle_{X\otimes_B Y} = \langle \eta, \varphi_Y(\langle \xi, \xi'\rangle_X)\eta'\rangle_Y
\]
for all $\xi,\xi'\in X$ and $\eta,\eta'\in Y$.
In particular, if $H$ is a $B$-representation, then we can view it as a $B$-$\C$-correspondence, so that $X\otimes_{B} H$ is an $A$-representation.
The tensor product construction is defined analogously for W*-correspondences.
If X is a $\fM$-$\fN$-W*-correspondence, then $X\otimes_{\fN} H$ is a normal representation of $\fM$ if $H$ is a normal representation of $\fN$.

\begin{definition}
	An (isometric) representation $(H,\rho_H, t)$ of an $A$-corres\-pondence (resp. W*-correspondence over $\fM$) $X$ on a Hilbert space $H$ consists of a non-degenerate $A$-representation (resp.~normal $\fM$-representation) $\rho_H \colon A \to \cB(H)$ and a linear map $t\colon X\to \cB(H)$ such that
	\begin{itemize}
		\item $\rho_H(a)t(\xi)\rho_H(a')  = t(\varphi_X(a)\xi\cdot a')$,
		\item $t(\xi)^*t(\eta) = \rho_H(\langle\xi,\eta\rangle)$
	\end{itemize}
	for all $a,a'\in A$ and $\xi,\eta\in X$.
	We often suppress the map $\rho_H$ from the notation and refer to $X$-representations as pairs $(H, t)$.
\end{definition}

\begin{lemma}\label{l:wstar-cont}
	Let $(H,t)$ be a representation of a correspondence $X$.
	The map $t$ is automatically completely contractive.
	If $X$ is a $\fM$-W*-correspon\-dence, then $t$ is also w*-continuous.
\end{lemma}
\begin{proof}
	We have
	\[
		{\|t(\xi)\|}^2 = \|t(\xi)^*t(\xi)\| = \|\rho(\langle\xi, \xi\rangle)\| \leq \|\langle \xi, \xi\rangle\| = {\|\xi\|}^2
	\]
	and the same computation works for matrices with entries in $X$, so $t$ is completely contractive.

	We now show the w*-continuity.
	The w*-closure of $t(X)$ in $\cB(H)$ is a Hilbert W*-module over $\rho(\fM)$.
	By viewing $\rho(\fM)$ as an ideal in $\fM$, we can regard $\overline{t(X)}^{w*}$ as a Hilbert $\fM$-module.
	Then, the bounded Hilbert $\fM$-module map $t\colon X\to \overline{t(X)}^{w*}$ is w*-continuous by \cite[Corollary 8.5.8]{BLM}.
\end{proof}

For a C*-correspondence, the Cuntz-Toeplitz algebra $\cT_X$ was defined by Pimsner~\cite{pimsner} as the universal C*-algebra generated by images of $A$ and $X$ in representations of $X$.
There are a $*$-homomorphism $\boldsymbol{\rho}\colon A \to \cT_X$ and a linear map $\boldsymbol{t}\colon X\to \cT_X$ with the following universal property.
For all representations $(H, t)$ of $X$, there is a unique representation $\rho_H\rtimes t\colon \cT_X\to \cB(H)$ such that $\rho_H = \rho_H\rtimes t\circ \boldsymbol{\rho}$ and $t = \rho_H\rtimes t\circ \boldsymbol{t}$.
We usually identify $A$ and $X$ with their images via $\boldsymbol{\rho}$ and $\boldsymbol{t}$.
We also identify $\cK(X)$ with its copy inside $\cT_X$ given by $\overline{\boldsymbol{t}(X)\boldsymbol{t}(X)^*}$.

There is a gauge action $\gamma$ of the unit circle $\mathbb{T} = \{z\in \C\colon |z|=1\}$ on $\cT_X$.
The action leaves $A$ invariant and acts on $X$ by $\gamma_z(\xi) = z\xi$ for all $\xi\in X$ and $z\in\mathbb{T}$

Let $I\subset \varphi_X^{-1}(\cK(X)) \subset A$ be an ideal.
Consider the ideal $\mathcal{J}_I\subset \cT_X$ generated by elements $\varphi_X(a)-\boldsymbol{\rho}(a)$ for $a\in I$.
The quotient $\cO_{X, I} = \cT_X/\mathcal{J}_I$ is called the $I$-relative Cuntz-Pimsner algebra.
An $X$-representation $(H, t)$ for which $\rho_H\rtimes t$ factors through $\cO_{X, I}$ is called $I$-relative Cuntz-Pimsner covariant.

The ideal $\cI_X = \varphi_X^{-1}(\cK(X))\cap {(\ker \varphi_X)}^{\perp}$ is called Katsura's ideal.
Katsura~\cite{katsura04} showed that the ideal $\mathcal{J}_{\cI_X}$ is the largest gauge-invariant ideal of $\cT_X$ trivially intersecting $\boldsymbol{\rho}(A)$ and the algebra $\cO_{X} = \cO_{X, \cI_X}$ is called just the Cuntz-Pimsner algebra of $X$.
Consequently, $X$-representations for which $\rho_H\rtimes t$ factors through $\cO_X$ are called Cuntz-Pimsner covariant.

A representation $(H,t)$ of $X$ induces an $A$-linear isometry $t_*\colon X\otimes_A H \to H$ by the formula $t_{*}(\xi\otimes v) = t(\xi) v$ for all $\xi \in X$ and $v\in H$.
Conversely, given an $A$-linear isometry $t_* \colon X\otimes_A H \to H$, we can define a representation $t \colon X\to \cB(H)$ by the formula $t(\xi)v = t(\xi\otimes v)$ for all $\xi\in X$ and $v\in H$.
\begin{proposition}[{\cite[Lemma 2.5]{muhly-solel-wstar}}]\label{p:tens-equiv-rep}
	Let $\rho\colon A\to \cB(H)$ be a representation.
	The assignments $t\leftrightarrow t_*$ above define mutually inverse bijections between the sets of $X$-representations $t\colon X\to \cB(H)$ and $A$-linear isometries $t_*\colon X\otimes_A H\to H$.
	The same applies to a W*-correspondence $X$ over $\fM$ and a normal representation $\rho\colon \fM\to \cB(H)$.
\end{proposition}
\begin{construction}\label{con:Fock}
	Let $M$ be an $A$-representation.
	Consider the \emph{$X$-Fock Hilbert space} $M^X = \bigoplus_{k \geq 0} X^{\otimes k}\otimes_A M$ where we set $X^{\otimes 0} = A$.
	Observe that $X\otimes_{A} M^X = \bigoplus_{k\geq 1} X^{\otimes k}\otimes_A M$ as an $A$-representation and denote by $(s_M)_*\colon X\otimes_{A} M^X \to M^X$ the tautological inclusion isometry.
	The isometry defines an $X$-representation $(M^X, s_M)$ by Proposition~\ref{p:tens-equiv-rep}.
	A representation which is unitarily equivalent to one of the form $(M^X, s_M)$ is called \emph{induced}.

	This construction applies also when $X$ is a W*-correspondence over a von Neumann algebra $\fM$ and $M$ is a normal $\fM$-representation.
\end{construction}

\begin{definition}\label{d:dilation}
	An \emph{(isometric) dilation} of an $X$-representation $(H, t)$ is a triple $(K,s,V)$, where $(K,s)$ is an $X$-representation and $V\colon H\to K$ is an $A$-linear isometry such that
	\[
		V^*s(\xi)V = t(\xi)
	\]
	for all $\xi\in X$.
	The dilation $(K, s, V)$ is called \emph{trivial} if $VH \subset K$ is $X$-reducing which means that $VH$ and $(VH)^\perp$ are invariant under $s(\xi)$ for all $\xi\in X$.

	The representation $(H, t)$ is called \emph{maximal} if it does not admit non-trivial dilations.
	The correspondence $X$ is called \emph{hyperrigid} if all its Cuntz-Pimsner covariant representations are maximal.

	The definitions above also apply to the case when $X$ is a W*-correspon\-dence without any modifications.
\end{definition}

\begin{remark}
	In this work, we assume all representations and dilations to be isometric.
	This assumption does not result in any loss of generality, as Muhly and Solel proved that every completely contractive representation admits an isometric dilation \cite{muhly-solel-dilation}.
	Therefore, the notion of maximality remains the same when considering only isometric dilations.
\end{remark}

\begin{lemma}\label{l:isometric-dil}
	If $(K, s, V)$ is a dilation of $(H, t)$, then
	\[
		s(\xi)V = Vt(\xi)
	\]
	for all $\xi\in X$.
\end{lemma}
\begin{proof}
	We have $s(\xi)V = VV^*s(\xi)V + (1 - VV^*)s(\xi)V = Vt(\xi) + (1- VV^*)s(\xi)V$ so it is enough to show that $T = (1-VV^*)s(\xi)V = 0$.
	We compute
	\[
		T^*T = V^*s(\xi)^*(1-VV^*)s(\xi)V = V^*\rho_K(\langle\xi, \xi\rangle)V - \rho_H(\langle\xi, \xi\rangle) = 0,
	\]
	where the last equality follows by the $A$-linearity of $V$.
\end{proof}

\begin{proposition}\label{p:uep-equiv}
	Let $X$ be a C*-correspondence over $A$.
	Consider the operator algebra $\cT_X^+\subset \cT_X$ generated by $A, X\subset \cT_X$ and the operator system $\mathcal{S}_X = X^* + A + X \subset \cT_X$.
	The following are equivalent for a representation $(H, t)$ of $X$:
	\begin{enumerate}
		\item $(H,t)$ is maximal;\label{p:uep-equiv:X}
		\item $\rho_H\rtimes t|_{\mathcal{S}_X}$ satisfies the UEP;\label{p:uep-equiv:S}
		\item $\rho_H\rtimes t|_{\cT_X^+ + (\cT_X^+)^*}$ satisfies the UEP.\label{p:uep-equiv:T}
	\end{enumerate}
\end{proposition}
\begin{proof}
	If $A$ is non-unital, then so is $\cT_X$ and by \cite[Proposition 2.4]{Salomon19} $\pi_{(H, t)}$ has the UEP with respect to $\cT_X^+$ or $\mathcal{S}_X$ if and only if its C*-unitization has the UEP.
	Therefore, we may assume that $A$ and  $\cT_X$ are unital.
	By \cite[Proposition 2.2]{arveson-note}, a representation has the UEP if and only if it is maximal.

	Suppose that \ref{p:uep-equiv:X} holds and $(H, t)$ is maximal.
	Since $A, X\subset \mathcal{S}_X$ and $A, X \subset \cT_X^+$, then any $\cT_X^+$-dilation or $\mathcal{S}_X$-dilation of $(H, t)$ is also an $X$-dilation and thus trivial.
	Therefore, \ref{p:uep-equiv:X} implies \ref{p:uep-equiv:T} and \ref{p:uep-equiv:S}.

	Analogously, $\mathcal{S}_X\subset \cT_X^+ + {(\cT_X^+)}^*$ and thus if a representation has UEP with respect to $\mathcal{S}_X$, then it also has the UEP with respect to $\cT_X^+$.
	This proves \ref{p:uep-equiv:S}$\Rightarrow$\ref{p:uep-equiv:T}.

	Finally, suppose that \ref{p:uep-equiv:T} holds.
	Let $(K, s, V)$ be a dilation of $(H, t)$.
	By an inductive application of Lemma~\ref{l:isometric-dil}, we have
	\[
		V^* s(\xi_1)s(\xi_2)\cdots s(\xi_n) V = t(\xi_1)t(\xi_2)\cdots t(\xi_n)
	\]
	for all $\xi_1,\xi_2,\dots,\xi_n\in X$.
	Since $\cT_X^+$ is linearly spanned by $A$ and elements of the form $\xi_1\cdots\xi_n$, it follows that $(\rho_K\ltimes s, V)$ is a $\cT_X^+$-dilation of $\rho_H\ltimes t$ and therefore it is trivial.
	We conclude that \ref{p:uep-equiv:T} implies \ref{p:uep-equiv:X}.
\end{proof}

\section{Second dual correspondence}\label{s:second-dual}
For a C*-algebra $A$, its second dual $\ssA$ is naturally endowed with the structure of a von Neumann algebra (see \cite[III.5.2.8]{Blackadar}).
It is universal in the following sense, if $\psi \colon A \to \fM$ is a
$*$-homomorphism into a von Neumann algebra $\fM$, then $\psi$ extends uniquely to a normal homomorphism $\tilde{\psi}\colon \ssA \to \fM$.
For example, if $\rho\colon A\to \cB(H)$ is a representation, then $\tilde\rho \colon \ssA \to \cB(H)$ is a normal representation.
Therefore, there is a one-to-one correspondence between representations of $A$ and normal representations of $\ssA$.

The multiplier algebra $\mathcal{M}(A)$ can be naturally embedded into $\ssA$ (see \cite{MultInSS}).
Consequently, if $B$ is another C*-algebra and $\phi\colon A\to \mathcal{M}(B) \subset \ssB$ is a $*$-homomorphism, then there is a unique normal homomorphism $\tilde\phi \colon \ssA\to \ssB$, which is unital if and only if $\phi(A)B = B$.

If $\psi\colon \fM \to \fN$ is a normal $*$-homomorphism, then there is a unique central projection $P_\psi \in Z(\fM)$ called \emph{the support projection of $\psi$} such that $\ker \psi = \fM(1-P_\psi)$ and $\psi(\fM) \cong \fM P_\psi$.
For a $*$-homomorphism $\phi\colon A\to \mathcal{M}(B)$, we define the support projection of $\phi$ to be $P_{\phi} = P_{\tilde \phi} \in \ssA$.
For an ideal $\iota_I\colon I\hookrightarrow B$, we use the notation $P_{I} = \tilde\iota_I(1) \in \ssB$.
It is the support projection of the associated map $B\to \mathcal{M}(I)$.

If $X$ is an $A$-$B$-correspondence, then by \cite[Proposition 8.5.17]{BLM}, $\ssX$ is a $W^*$-module over $\ssB$ such that $\cL(\ssX) = \cK(X)^{**}$.
The map $\varphi_X \colon A\to \cL(X) = \mathcal{M}(\cK(X))$ lifts uniquely to the normal map $\varphi_{\ssX} = \tilde\varphi_X  \colon \ssA \to \cL(\ssX)$ which endows $\ssX$ with a structure of a W*-correspondence.
We define the support projection of $X$ to be $P_X = P_{\varphi_{X}}\in \ssA$.

\begin{lemma}\label{l:ssPI}
	Let $X$ be an $A$-$B$-correspondence and $I$ be an ideal.
	Then, $\Ss{(\varphi(I)X)} = \tilde\varphi(P_I)\ssX \subset \ssX$.
	In particular, if $H$ is an $A$-representation, then $\rho_H(I)H = \tilde\rho_H(P_I)H$.
\end{lemma}
\begin{proof}
	The subspace $\Ss{(\varphi(I)X)}$ is a w*-closure of $\varphi(I)X$ in $\ssX$.
	Since $P_II = I$ in $\ssA$, it follows that $\Ss{(\varphi(I)X)}\subset \tilde\varphi(P_I)\ssX$.
	On the other hand, $P_I$ is in the w*-closure of $I$ in $\ssA$, which implies $\Ss{(\varphi(I)X)}\subset \tilde\varphi(P_I)\ssX$.
	This proves the first claim.
	The second claim follows from the first one by considering $H$ as an $A$-$\C$-correspondence and from the fact that Hilbert spaces are reflexive.
\end{proof}

Thanks to the universal property of the second dual, there is a bijection between $*$-representations $\rho\colon A\to B(H)$ and normal $*$-representations $\tilde\rho \colon \ssA \to \cB(H)$.
This can be extended to correspondences as follows.
Since $X$ is w*-dense in $\ssX$, every linear map $t\colon X\to \cB(H)$ extends uniquely to a w*-continuous map $\tilde t\colon \ssX\to \cB(H)$.
Because the left and right actions of $\ssA$ on $\ssX$ and the inner product are w*-continuous, the maps $\tilde\rho_H$ and $\tilde t$ define a representation of $\ssX$ on $H$.
On the other hand, we can recover $\rho_H$ and $t$ by restricting to w*-dense subspaces $A$ and $X$.
Thus, we have proved the following.
\begin{lemma}
	The assignments $(H, \rho_H, t) \leftrightarrow (H, \tilde\rho_h, \tilde t)$ define a bijection between the sets of $X$-representations and $X^{**}$-representations on $H$.
\end{lemma}

The following result reduces the study of dilations for C*-correspon\-dences to W*-correspondences.
This is the key result to prove Theorem~\ref{t:criterion}, which is the main theorem.
\begin{proposition}\label{p:ss-equiv}
	A triple $(K, s, V)$ is a dilation of $(H, t)$ if and only if $(K, \tilde s, V)$ is a dilation of $(H, \tilde t)$.
	In particular, $(H, t)$ is maximal if and only if $(H, \tilde t)$ is maximal.
\end{proposition}
\begin{proof}
	By construction, we have $\tilde s (\xi) = s(\xi)$ and $\tilde t(\xi) = t(\xi)$ for all $\xi \in X$.
	Therefore, if $(K, \tilde s, V)$ is a dilation of $(H,\tilde t)$, then $(K, s, V)$ is a dilation of $(H, t)$.

	On the other hand, suppose that $(K,s, V)$ is a dilation of $(H,t)$.
	The multiplication of bounded operators is separately w*-continuous.
	Therefore, if $\xi\in \ssX$ is an arbitrary element and $\{\xi_\lambda\}_{\lambda\in \Lambda}$ is a net in $X$ which converges to $\xi$ in the w*-topology, then
	\[
		V^*\tilde s(\xi)V = \operatorname*{w*-lim}_{\lambda\in \Lambda} V^* s(\xi_\lambda)V = \operatorname*{w*-lim}_{\lambda\in \Lambda} t(\xi_\lambda) = \tilde t(\xi).
	\]
	It follows that $(K, \tilde s, V)$ is a dilation of $(H,\tilde t)$.

	We conclude that dilations of $(H,t)$ are in bijection with dilations of $(H,\tilde t)$.
	In particular, $(H,t)$ is maximal if and only if $(H,\tilde t)$ is maximal.
\end{proof}

\section{Maximality criteria}
In this section we develop the general theory of dilations for correspondences.
We first deal with W*-correspondences and then proceed to the case of C*-correspondences using the second dual technique.

Consider a W*-correspondence $X$ over a W*-algebra $\fM$.
We assume all representations of $\fM$ to be normal.
Let $(H, t)$ be a representation of a correspondence $X$.
Denote $H_0 = H\ominus \overline{t(X)H}$ which is an $\fM$-reducing subspace.
The map $t_*\colon X\otimes_\fM H\to \overline{t(X)H}\oplus H_0 = H$ can be now written as a column block-matrix
\[
	t_* = \begin{bmatrix}
		Qt_* \\
		0
	\end{bmatrix} \colon X\otimes_\fM H \to \overline{t(X)H} \oplus H_0,
\]
where $Q$ is the orthogonal projection $H\twoheadrightarrow \overline{t(X)H}$.
Dilations of the representation $(H, t)$ correspond to dilations of the map $t_*$ and the above matrix form helps to classify those dilations.

\begin{construction}\label{con:dilation}
	Consider an $\fM$-representation $L$ together with a pair of bounded $\fM$-linear operators $s_0\colon X\otimes_\fM L \to H_0$ and $s_1\colon X\otimes_\fM L \to L$ satisfying $s_0^*s_0 + s_1^*s_1 = 1_{X\otimes_\fM L}$.
	For $K = H\oplus L$, we can form an $\fM$-linear map $s_* \colon X\otimes_\fM K = X\otimes_\fM H \oplus X\otimes_\fM L \to K = \overline{t(X)H}\oplus H_0 \oplus L$ by the block matrix
	\[
		s_* = \begin{bmatrix}
			Qt_* & 0   \\
			0    & s_0 \\
			0    & s_1
		\end{bmatrix} \colon (X\otimes_\fM H)\oplus (X\otimes_\fM L) \to \overline{t(X)H} \oplus H_0 \oplus H.
	\]
	Due to the relation on $s_0$ and $s_1$, the map $s_*$ is an isometry and thus defines an isometric $X$-representation $(K, s)$.
	Together with the tautological inclusion $V\colon H \hookrightarrow K$, the triple $(K, s, V)$ is a dilation of $(H, t)$, which is trivial if and only if $s_0 = 0$.
\end{construction}

\begin{lemma}\label{l:dil-constr}
	Every dilation of $(H, t)$ is equivalent to $(K,s,V)$ from Construction \ref{con:dilation} for some $L, s_0, s_1$.
\end{lemma}
\begin{proof}
	Let $(K, s, V)$ be a dilation of $(H, t)$.
	Identify $H$ with $VH$ and set $L = K\ominus H$.
	Since it is a dilation, the map $s_*$ has the block form
	\[
		s_* = \begin{bmatrix}
			Qt_* & s_2 \\
			0    & s_0 \\
			s_3  & s_1
		\end{bmatrix} \colon (X\otimes_\fM H)\oplus (X\otimes_\fM L) \to \overline{t(X)H} \oplus H_0 \oplus L,
	\]
	where $s_i$ are $\fM$-linear operators between corresponding Hilbert spaces.

	To prove the lemma, we must show that $s_2 = 0$, $s_3 = 0$, and $s_0^*s_0 + s_1^*s_1 = 1_{X\otimes_\fM L}$.
	Everything follows from the condition that $s_*$ is an isometry.
	Indeed, we have
	\begin{multline*}
		\begin{bmatrix}
			1_{X\otimes_\fM H} & 0                  \\
			0                  & 1_{X\otimes_\fM L}
		\end{bmatrix} =
		s_*^*s_* = \begin{bmatrix}
			t_*^*Q^*Qt_* + s_3^*s_3 & t_*^*Q^*s_2+s_3^*s_1           \\
			s_2^*Qt_* + s_1^*s_3    & s_0^*s_0 + s_1^*s_1 + s_2^*s_2
		\end{bmatrix}.
	\end{multline*}
	Since $Qt_*$ is a unitary $X\otimes_\fM H\to \overline{t(X)H}$, we have $s_3^*s_3 = 1_{X\otimes_\fM H} - t^*_*Q^*Qt_* = 0$ and thus $s_3=0$.
	Consequently, in the lower left corner we have $0 = s_2^*Qt_*$, which implies $s_2=0$.
	Finally, the condition on (2,2)-entry is $s_0^*s_0+s_1^*s_1 = 1_{X\otimes_\fM L}$ and the dilation $(K, t, V)$ is of the form of Construction~\ref{con:dilation} for the data $L, s_0, s_1$.
\end{proof}

Recall that for a W*-correspondence $X$ over $\fM$ we define the support projection $P_X\in Z(\fM)$ as the unique central projection satisfying $\ker \varphi_X = \fM(1-P_X)$.
\begin{definition}\label{d:p-relative}
	Let $X$ be a W*-correspondence over $\fM$ and let $P\in Z(\fM)$ be a central projection.
	A representation $(H, t)$ is called \emph{$P$-relative Cuntz-Pimsner covariant} if $\rho_H(P)H\subset \overline{t(X)H}$.
	A $P_{X}$-relative Cuntz-Pimsner covariant representation is called \emph{fully Cuntz-Pimsner covariant}.

	If $X$ is a C*-correspondence over $A$, then a representation $(H,t)$ is called $P$-relative Cuntz-Pimsner covariant for $P\in Z(A^{**})$ or fully Cuntz-Pimsner covariant if the associated representation $(H,\tilde t)$ of $X^{**}$ is respectively $P$-relative or fully Cuntz-Pimsner covariant.
\end{definition}
Since  $\varphi(P_X)X = X$ for a W*-correspondence $X$, we have $\rho(P_X)t(X) = t(X)$ for any representation $(H, t)$ and thus $\overline{t(X)H}\subset \rho(P_X)H$.
Therefore, the $X$-representation is fully Cuntz-Pimsner covariant if and only if $\overline{t(X)H} = \rho(P_X)H$.

If $X$ is a C*-correspondence and $(H, t)$ is an $X$-representation, then by w*-continuity $\overline{t(X)H} = \overline{\tilde t(\ssX)H}$.
In particular, the representation is fully Cuntz-Pimsner covariant if and only if $\overline{t(X)H} = \tilde\rho(P_X)H$.

\begin{theorem}\label{t:criterion}
	A representation of a C*- or W*-correspondence is maximal if and only if it is fully Cuntz-Pimsner covariant.
\end{theorem}
\begin{proof}
	By Proposition \ref{p:ss-equiv}, it is enough to prove the theorem for a W*-correspondence $X$ over $\fM$.
	Let $H_0 = H\ominus \overline{t(X)H}$, which is an $\fM$-reducing subspace.
	Then, the representation is fully Cuntz-Pimsner covariant if and only if $\rho(P_X) H_0 = 0$.

	Suppose that $(H, t)$ is not maximal.
	Then, by Lemma~\ref{l:dil-constr}, there exists an $\fM$-representation $L$ together with $\fM$-linear maps $s_0$ and $s_1$ such that $s_0\colon X\otimes_{\fM} L \to H_0$ is nonzero.
	The projection $P_X$ acts by identity on $X$ and thus on $X\otimes_{\fM} L$.
	Therefore, since $s_0$ is $\fM$-linear, the projection $P_X$ also acts by identity on $\overline{\operatorname{Im} s_0} \subset H_0$.
	Consequently, $\rho(P_X)H_0 \neq 0$ and the representation is not fully Cuntz-Pimsner covariant.

	Suppose now that $\rho(P_X)H_0 \neq 0$.
	Let $\rho_M\colon \fM \hookrightarrow \cB(M)$ be a faithful normal representation on a Hilbert space $M$.
	Then, by the Morita-Rieffel theory for von Neumann algebras (see \cite[8.5.12]{BLM}), $X\otimes_{\fM} M$ is a faithful representation of $\cL(X)$ and thus of $\fM P_X \cong \varphi(\fM)$.
	Since $\rho(P_X)H_0$ is a non-trivial representation of $\fM P_X$, from the representation theory of von Neumann algebras it follows that $\rho(P_X)H_0$ is quasi-equivalent to a subrepresentation of $X\otimes_{\fM} M$.
	Therefore, there exists a non-zero $\fM$-linear contraction $s_0'\colon X\otimes_\fM M \to H_0$.

	Let $L = M^X$ be the $X$-Fock space of $M$ and ${(s_M)}_*\colon X\otimes_{\fM} L\to L$ be the corresponding $\fM$-linear isometry (see Construction~\ref{con:Fock}).
	Define $s_0\colon X\otimes_{\fM} L \to H_0$ to be equal to $s_0'$ on $X\otimes_\fM M\subset X\otimes_\fM L$ and zero on the orthogonal complement.
	Also, set $s_1 = {(s_M)}_*\sqrt{1 - s_0^*s_0} \colon X\otimes_{\fM} L \to L$.
	We have $s_0^*s_0 + s_1^*s_1 = 1_{X\otimes_{\fM} L}$ and Construction~\ref{con:dilation} gives a dilation $(K, s, V)$ with $K = H\oplus L$ and
	\[
		s_* = \begin{bmatrix}
			Qt_* & 0   \\
			0    & s_0 \\
			0    & s_1
		\end{bmatrix} \colon (X\otimes_\fM H)\oplus (X\otimes_\fM L) \to \overline{t(X)H} \oplus H_0 \oplus L.
	\]
	The dilation is non-trivial since $s_0$ is non-zero by construction.
	We conclude that full Cuntz-Pimsner covariance is equivalent to maximality.
\end{proof}
\begin{example}\label{ex:dor-on-salomon}
	We analyze topological graphs and quivers in Section~\ref{s:quivers}.
	However, we already have enough tools to describe maximality for discrete graphs.

	Let $(E^0, E^1, r, s)$ be a countable discrete graph and $X = X(E)$ be its graph correspondence over $c_0(E^0)$ (see \cite[Chapter 8]{raeburn-graph}).
	Representations $(H,t)$ of $X$ are in bijection with Cuntz-Krieger families $\{p_v, s_e\}_{v\in E^0, e\in E^1}$.
	Moreover, we have $\overline{t(X)H} = \bigoplus_{e\in E^1} s_es_e^*H$.

	We have $\Ss{c_0(E^0)} = \ell^\infty(E^0)$ and $\ker \tilde\varphi_X = \ell^{\infty}(E^0\setminus r(E^0))$.
	Therefore, the support projection $P_X$ equals to the sum of projections corresponding to non-sink vertices $r(E^0)$.
	Consequently, the representation is fully Cuntz-Pimsner covariant if and only if $\sum_{e\in E^1} s_es_e^* = \sum_{v\in r(E^0)} p_v$, where the sum is taken in the w*-topology.
	This coincides with the notion of full Cuntz-Krieger families by Dor-On and Salomon~\cite{doron-salomon}.
	Therefore, Theorem~\ref{t:criterion} is a generalization of \cite[Theorem 3.5]{doron-salomon}.
\end{example}

\begin{corollary}\label{cor:ind-max}
	Let $M$ be a representation of $A$.
	The induced representation $(M^X, s_M)$ is maximal if and only if $\tilde{\rho}_M(P_X) = 0$.
\end{corollary}
\begin{proof}
	By construction, $M = M^X \ominus \overline{s_M(X)M^X}$ as an $A$-representation.
	Therefore, $\tilde{\rho}_{M^X}(P_X)M^X\subset \overline{s_M(X)M^X}$ if and only if $\tilde{\rho}_M(P_X) = 0$.
\end{proof}

\begin{lemma}\label{l:cp-ideal-covariant}
	A representation $(H, t)$ of a C*-correspondence $X$ is $P_I$-relative Cuntz-Pimsner covariant if and only if it is $I$-relative covariant in the sense of Katsura.
	In particular, the representation is Cuntz-Pimsner covariant if and only if it is $P_{\cI_X}$-relative Cuntz-Pimsner covariant, where $\cI_X$ is Katsura's ideal.
\end{lemma}
\begin{proof}
	If $H$ is an $A$-representation, then we have $\rho(\cI_X)H = \tilde\rho(P_{\cI_X})H$ by Lemma~\ref{l:ssPI}.
	The statement follows from \cite[Proposition 2.16]{MeyerSehnemBicat} for $B = \cB(H)$.
\end{proof}

Let $\mathcal{A}$ be a non-self-adjoint subalgebra of a C*-algebra $B$.
In \cite{arveson69}, Arveson defined the Shilov ideal $J$ of $\mathcal{A}$ to be the largest ideal of $B$ such that the quotient $B\twoheadrightarrow B/J$ is completely isometric on $\mathcal{A}$.
The quotient $C^*_{env}(\mathcal{A})\coloneqq B/J$ is called the C*-envelope of $\mathcal{A}$.
The existence of the Shilov ideal and of the C*-envelope was established by Hamana~\cite{hamana}.
\begin{corollary}[{\cite[Theorem 3.7]{katsoulis-kribs}}]\label{cor:katsoulis-kribs}
	Maximal representations of a C*-correspondence $X$ are Cuntz-Pimsner covariant in the sense of Katsura.
	In particular, $C^*_{env}(\cT_X^+) = \cO_X$.
\end{corollary}
\begin{proof}
	The Cuntz-Pimsner covariance is the same as the $P_{\cI_X}$-relative Cuntz-Pimsner covariance by Lemma~\ref{l:cp-ideal-covariant}.
	The restriction $\varphi_{X}|_{\cI_X}\colon \cI_X \to \cK(X)$ is injective and thus $\tilde\varphi_{X}|_{\Ss{\cI_X}}\colon \Ss{\cI_X} = \ssA P_{\cI_X} \to \cL(\ssX)$ is also injective.
	Therefore, we have $P_{\cI_X} \leq P_X$ and, hence, fully Cuntz-Pimsner covariant representations are Cuntz-Pimsner covariant.
	By the proof of \cite[Theorem 4.1]{dritschel-mc}, there is a completely isometric maximal representation $\pi$ of $\cT_X^+$ such that the C*-algebra generated by its image is the C*-envelope $C^*_{env}(\cT_X^+)$ (see also \cite[Section 3]{arveson-note}).
	Since maximal representations are Cuntz-Pimsner covariant, the representation $\pi$ factors through $\cO_X$ and thus $C^*_{env}(\cT_X^+)$ is a quotient of $\cO_X$.

	Moreover, if $J\subset \cO_X$ is the Shilov ideal and $z\in \mathbb{T}$, then the quotient $\cO_X/\gamma_z(J)$ is completely isometric on $\cT_X^+$ and thus $\gamma_z(J)\subset J$.
	Therefore, the Shilov ideal is gauge-invariant.
	However, any non-zero gauge-invariant ideal of $\cO_X$ intersects non-trivially with $A\subset \cT_X^+$ by \cite[Proposition 7.14]{katsura04}, and thus the quotient is not completely isometric on $\cT_X^+$.
	Therefore, $\cO_X$ is the C*-envelope of $\cT_X^+$.
\end{proof}

\begin{corollary}[{\cite[Theorem 1.2]{kim}}]\label{cor:kim}
	A C*-correspondence $X$ is hyperrigid if and only if $\varphi(\cI_X)X=X$.
\end{corollary}
\begin{proof}
	We have seen in the proof of Corollary~\ref{cor:katsoulis-kribs} that $P_{\cI_X} \leq P_X$.
	By Lemma~\ref{l:ssPI}, we have $\tilde\varphi(P_{\cI_X})\ssX = \Ss{(\varphi(\cI_X)X)} \subset \ssX$.
	Therefore, $P_{\cI_X} = P_X$ if and only if $\varphi(\cI_X)X = X$.
	In particular, if the latter holds, then all Cuntz-Pimsner covariant representations are automatically fully Cuntz-Pimsner covariant and $X$ is hyperrigid.

	Conversely, if $P_{\cI_X} < P_X$, then there is an $A$-representation $M$ such that $\tilde\rho_M(P_{\cI_X}) = 0$ but $\tilde\rho_M(P_X) = 1$.
	The first condition ensures that the induced representation $(M^X, s_M)$ is Cuntz-Pimsner covariant and the second condition implies that it is not maximal by Corollary~\ref{cor:ind-max}.
	Therefore, $X$ is not hyperrigid if $\varphi(\cI_X)X\neq X$.
\end{proof}

We say that a C*-correspondence $X$ is \emph{almost proper}, if the ideal $\varphi^{-1}(\cK(X))\subset A$ acts non-degenerately on $X$.
For an almost proper correspondence, we give an equivalent characterization of the full Cuntz-Pimsner covariance and thus of maximality without any reference to the second dual.
\begin{proposition}\label{p:almost-proper}
	A representation $(H,t)$ of an almost proper C*-cor\-respondence $X$ is fully Cuntz-Pimsner covariant if and only if the identity $\rho\left(\varphi^{-1}(\cK(X))\right)H = \overline{t(X)H}\oplus \rho(\ker \varphi)H$ holds.
\end{proposition}
\begin{proof}
	Let $I = \varphi^{-1}(\cK(X))$ and $J = \ker\varphi \subset I$.
	The inclusion $I/J \hookrightarrow \cK(X)$ induces the inclusion $\Ss{(I/J)} = \ssA (P_I-P_J)\hookrightarrow \cL(\ssX)$.
	Therefore, we have $P_I - P_J \leq P_X$.

	By Lemma~\ref{l:ssPI} and the assumption that $X$ is almost proper, we have $\tilde\varphi(P_I)\ssX = \ssX$ and $\tilde\varphi(P_J)\ssX = 0$.
	Therefore, $P_I-P_J$ acts by the identity on $X$ which implies $P_X \leq P_I - P_J$ and thus $P_X = P_I - P_J$.
	We conclude that a representation $(H, t)$ is fully Cuntz-Pimsner covariant if and only if $\overline{t(X)H} = \tilde\rho(P_I - P_J)H = \rho(I)H\ominus\rho(J)H$.
\end{proof}

\section{Boundary representations}\label{s:boundary}
Let $X$ be a C*-correspondence over $A$.
An irreducible representation of $\cT_X$ is called boundary if its restriction to the operator system $\cT_X^+ + (\cT_X^+)^*$ satisfies the UEP.
The noncommutative Choquet boundary $\partial_C \cT_X^+$ is the set of unitary equivalence classes of boundary representations.
Since any boundary representation factors through the C*-envelope, $\partial_C \cT_X^+$ is naturally a subset of $\widehat{\cO}_X$.
The goal of this section is to compute $\partial_C X\coloneqq \partial_C \cT_X^+$.

We say that a representation $(H, t)$ of a C*-correspondence $X$ is irreducible if the induced representation $\rho_H\rtimes t$ of $\cT_X$ is irreducible.
By Proposition~\ref{p:uep-equiv}, any boundary representation of $\cT_X$ is induced, up to unitary equivalence, by an irreducible maximal representation of $X$.
Therefore, we also refer to irreducible maximal representations of $X$ as boundary representations.

A representation $(H,t)$ is called fully coisometric if $\overline{t(X)H} = H$.
Fully coisometric representations are automatically fully Cuntz-Pimsner covariant and thus they are maximal.
By the Wold decomposition~\cite{muhly-solel-wold} for correspondences, every $X$-representation is a direct sum of fully coisometric and induced parts.
Therefore, irreducible representations of $\cT_X$ and $\cO_X$ are either fully co-isometric or induced:
$\widehat{\cT}_X = {(\widehat{\cT}_X)}_{\mathrm{fc}} \sqcup {(\widehat{\cT}_X)}_{\mathrm{ind}}$ and $\widehat{\cO}_X = {(\widehat{\cO}_X)}_{\mathrm{fc}} \sqcup {(\widehat{\cO}_X)}_{\mathrm{ind}}$.

Since fully coisometric representations are Cuntz-Pimsner covariant, we have ${(\widehat{\cT}_X)}_{\mathrm{fc}} = {(\widehat{\cO}_X)}_{\mathrm{fc}}$ and we denote this set by $\widehat{X}_{\mathrm{fc}}$.
By the maximality of fully coisometric representations, we have $\widehat{X}_{\mathrm{fc}} \subset\partial_C X$.
To describe $\partial_C X$ we thus only need to classify the induced boundary representations.

\begin{lemma}\label{l:ind-irr}
	An induced representation $(M^X, s_M)$ is irreducible if and only if $M$ is an irreducible representation of $A$.
	Therefore, there is a bijection $\operatorname{ind}\colon \widehat{A} \to {(\widehat{\cT}_X)}_{\mathrm{ind}}$ such that $\operatorname{ind}\widehat{A/\cI_X} = {(\widehat{\cO}_X)}_{\mathrm{ind}}$.
\end{lemma}
\begin{proof}
	Suppose that $M = M'\oplus M''$ is not irreducible.
	Then, $(M')^X \subset M^X$ is a proper $\cT_X$-reducing subspace and the induced representation on $M^X$ is not irreducible.
	Conversely, suppose that $M^X$ is not irreducible and consider an $\cT_X$-reducing subspace $H\subset M^X$ having orthogonal projection $P$.
	We claim that $PM\subset M$.
	Indeed, $M = M^X\ominus s_M(X)M^X$ can be characterized as a subspace of vectors $v$ having $\langle v, s_M(\xi)w\rangle = 0$ for all $\xi\in X$ and $w\in M^X$.
	By $\cT_X$-reducibility, we have $\langle Pv, s_M(\xi)w\rangle = \langle v, s_M(\xi)P w\rangle = \langle v, s_M(\xi)w \rangle = 0$ and thus $Pv\in M$.

	Since $PM$ is an $A$-subrepresentation of an irreducible representation $M$, we have either $PM=0$ or $PM=M$.
	If $PM=M\subset H$, then $H = M^X$ since subspaces $t(X)^nM$ for $n\geq 0$ span the whole $M^X$.
	Otherwise, we have $PM = 0$ which implies $M = (1-P)M \subset H^\perp$ and $H^\perp = M^X$ for the same reason.
	We conclude that $M^X$ does not contain any proper $\cT_X$-reducing subspace.
\end{proof}

\begin{definition}\label{d:atomic}
	The \emph{atomic spectrum $\sigma_a(X) \subset \widehat{A}$ of $X$} consists of those irreducible representations $(M, \pi_M) \in \widehat{A}$ with $\tilde{\pi}_M(P_X) = 1_H$ or, equivalently, $P_{\pi_M}\leq P_X$.
\end{definition}
The term \emph{atomic spectrum} is motivated by Proposition~\ref{p:hyperrigidity-of-motimesa} which says that the atomic spectrum of the correspondence constructed from a representation coincides with the set of isomorphism classes of irreducible subrepresentations (atoms).

Note that $\widehat{\cI}_X\subset \sigma_a(X)$ since $P_{\cI_X}\leq P_X$ by the proof of Corollary~\ref{cor:katsoulis-kribs}.
It turns out that the equality holds exactly when the noncommutative Choquet boundary coincides with the noncommutative Shilov boundary.
\begin{theorem}\label{t:choquet}
	Let $X$ be a C*-correspondence over a C*-algebra $A$.
	The Choquet boundary of $X$ consists of all irreducible representations of $\cO_X$ except those induced from elements of the atomic spectrum:
	\[
		\partial_C X = \widehat{\cO}_X\setminus \operatorname{ind}\left(\sigma_a(X)\right) = \widehat{X_{\mathrm{fc}}} \sqcup \operatorname{ind}\left(\widehat{A}\setminus \sigma_a(X)\right).
	\]
	In particular, $\partial_C X = \widehat{\cO}_X$ if and only if $\sigma_a(X) = \widehat{\cI}_X$.
\end{theorem}
\begin{proof}
	Let $(H, t)\in \widehat{\cO}_X$ be an irreducible representation of $X$.
	It is enough to prove that $(H, t)$ is not maximal if and only if it is induced from a representation in the atomic spectrum.

	By Wold decomposition, $(H, t)$ is either fully coisometric or induced.
	If it is fully coisometric, then it is maximal.
	Suppose that it is induced by an $A$-representation $M$ which is irreducible by Lemma~\ref{l:ind-irr}.
	Since $P_X$ is a central projection, $\tilde{\rho}_M(P_X)\in \rho_M(A)'' = \C$ and thus either $\tilde{\rho}_M(P_X) = 0$ or $M\in \sigma_a(X)$.
	By Corollary~\ref{cor:ind-max}, $(H, t)$ is maximal if and only if $\tilde{\rho}_M(P_X) = 0$.
	We conclude that $(H, t)$ is non-maximal exactly when $M\in \sigma_a(X)$.
\end{proof}
Arveson~\cite{arveson-hyperrigidity} conjectured that an operator system is hyperrigid if and only if all irreducible representations of its C*-envelope are boundary.
The author together with Dor-On~\cite{hyp-counter} found an elementary type I counterexample to the conjecture.
However, it is still interesting to understand for which classes of operator systems the conjecture holds.

We say that a C*-correspondence $X$ violates the hyperrigidity conjecture if $\partial_C X = \widehat{\cO}_X$ but $X$ is not hyperrigid.
Theorem~\ref{t:choquet} has an immediate consequence that $X$ violates the hyperrigidity conjecture if and only if $\sigma_a(X) = \widehat{\cI}_X$ while $\varphi(\cI_X)X\neq X$.
We will see that there are plenty of correspondences having this property.

Let $L$ be a Hilbert space together with an $A$-representation $\rho_L\colon A\to \cB(L)$.
We can view $L$ as an $A$-$\C$-correspondence and $A$ as a $\C$-$A$-corres\-pondence to produce an $A$-correspondence $L\otimes_\C A$.
Kumjian~\cite{Kumj04} showed that the Cuntz-Pimsner algebra $\cO_{L\otimes_\C A}$ is purely infinite and simple provided that $\rho_L(A)\cap \cK(L) = 0$.
\begin{proposition}\label{p:hyperrigidity-of-motimesa}
	The atomic spectrum $\sigma_a(L\otimes_\C A)$ coincides with the atomic spectrum of $L$, that is, the set of isomorphism classes of irreducible subrepresentations of $L$.
	Consequently, if $L$ does not contain any irreducible subrepresentations, then $L\otimes_\C A$ violates the hyperrigidity conjecture.
\end{proposition}
\begin{proof}
	Let $X = L\otimes_\C A$.
	We have $\ker \tilde{\varphi}_X = \ker \tilde{\rho}_L$.
	Therefore, we also have $\tilde{\rho}_L(A^{**}) = \rho_L(A)'' \cong A^{**}P_X$.
	By the elementary theory of von Neumann algebras, $P_X$ acts non-degenerately on an irreducible representation of $A$ if and only if the representation is contained in $L$.

	If $L$ does not contain any irreducible subrepresentations, then $\sigma_a(X) = \emptyset$ and $\partial_C X = \widehat{\cO}_X$ by Theorem~\ref{t:choquet}.
	On the other hand, the induced representation $L^X$ of $X$ is not maximal by Corollary~\ref{cor:ind-max} and thus $X$ is not hyperrigid.
	Alternatively, one could show that $\cI_X = 0$ and apply Kim's hyperrigidity criterion.
\end{proof}

\begin{example}
	Let $A = C^*(\Z)$ and $X = \ell^2(\Z, C^*(\Z))$ with basis ${(e_k)}_{k\in \Z}$.
	The group $\Z$ acts on $X$ by the automorphism $e_k\mapsto e_{k+1}$, which defines the correspondence structure $\varphi_X\colon C^*(\Z) \to \cL(X)$.
	An $X$-representation on $H$ is equivalent to isometries $S_k = t(e_k)$ with pairwise orthogonal ranges and a unitary $\rho(\delta_1) = U$ such that $US_k = S_{k+1}$.
	Therefore, the purely infinite simple C*-algebra $\cO_X$ is obtained by attaching a unitary satisfying the above relation to the infinite Cuntz algebra $\cO_\infty$.

	The correspondence $X$ is isomorphic to the tensor product $\ell^2(\Z) \otimes_{\C} C^*(\Z)$, where the group $\Z$ acts on $\ell^2(\Z)$ by the bilateral shift.
	Since the bilateral shift does not have any eigenvectors, $X$ violates the hyperrigidity conjecture by Theorem~\ref{p:hyperrigidity-of-motimesa}.
\end{example}

\begin{example}
	Let $\theta\in (0,1)$ be an irrational number and let $A_\theta = \langle u, v\colon u^*=u^{-1},v^*=v^{-1}, uv = e^{2\pi i \theta}vu\rangle$ be the irrational rotation algebra.
	The algebra $A_\theta$ is a simple C*-algebra with $\widehat{A}_\theta = \mathbb{T}/\Z\theta$.
	In particular, it is not of Type I.

	Let $M = \ell^2(\Z\times \Z)$ with basis ${(\xi_{k, j})}_{k,j\in \Z}$.
	Define unitaries $U, V$ on $M$ by $U\xi_{k,j} = \xi_{k+1,j}$ and $V\xi_{k, j} = e^{-2\pi i\theta k}\xi_{k, j+1}$.
	Setting $\rho(u) = U$ and $\rho(v) = V$ defines a representation $A_\theta$ on $M$.
	This representation satisfies the assumption of Theorem~\ref{p:hyperrigidity-of-motimesa} since in any irreducible representation of $u$ has an eigenvector, whereas $U$ is a direct sum of bilateral shifts and does not contain any.
	Therefore, $X = M\otimes_{\C} A_\theta$ violates the hyperrigidity conjecture.
\end{example}

\begin{proposition}\label{p:conjecture-proper}
	If $X$ is a proper correspondence, then $\sigma_a(X) = \widehat{\varphi_X(A)} = \widehat{A}\setminus\widehat{\ker \varphi_X}$.
	Consequently, Arveson's hyperrigidity conjecture holds for proper correspondences.
\end{proposition}
\begin{proof}
	We have $P_X = 1 - P_{\ker \varphi_X}$ and thus $M\in \sigma_a(X)$ if and only if $\tilde\rho_M(P_{\ker\varphi_X}) = 0$, which is equivalent to $\ker\varphi_X \subset \ker \rho$.
	Therefore, $\sigma_a(X)$ consists of the irreducible representations which factor through $A\twoheadrightarrow \varphi_X(A) \cong A/\ker\varphi_X$.

	Suppose now that $\partial_C X = \widehat{\cO}_X$.
	Therefore, we have $\widehat{\varphi_X(A)} = \partial_C X = \widehat{\cI}_X$ by Theorem~\ref{t:choquet}.
	Since $\cI_X \cong \varphi_X(\cI_X)$ is an ideal in $\varphi_X(A)$, this implies $\varphi_X(\cI_X) = \varphi_X(A)$.
	Therefore, the ideal $\cI_X$ acts non-degenerately on $X$ and the correspondence is hyperrigid.
	We conclude that Arveson's conjecture holds in this case.
\end{proof}

Suppose now that $A$ is a separable Type I C*-algebra.
It is equivalent to the fact that any irreducible representation is uniquely determined up to isomorphism by its kernel by Glimm's Theorem~\cite[Theorem 1]{glimm}.
Therefore, there is a canonical bijection between $\widehat{A}$ and $\operatorname{Prim}A$ and we identify those sets.
Moreover, if $A = C_0(Y)$, then we also identify $\widehat{A}$ and $\operatorname{Prim}A$ with $Y$.

Let $\mathfrak{P} \in \operatorname{Prim}A$ be some primitive ideal and let $\pi_{\mathfrak{P}}$ be the corresponding irreducible representation on $K_{\mathfrak{P}}$.
Again by Glimm's Theorem~\cite[Theorem 1]{glimm}, the irreducible representation on $K_{\mathfrak{P}}$ is GCR, meaning that $\mathbb{K}(K_{\mathfrak{P}}) \subset \pi_{\mathfrak{P}}(A)$.
Then, $J_0(\mathfrak{P}) = \pi_{\mathfrak{P}}^{-1}(\mathbb{K}(K_{\mathfrak{P}}))$ is the smallest ideal of $A$ properly containing $\mathfrak{P}=\ker \pi_{\mathfrak{P}}$.
If $A$ is commutative, then every primitive ideal is maximal and $J_0(\mathfrak{P}) = A$.
\begin{lemma}
	The support projection $P_{\pi_{\mathfrak{P}}}$ equals to $P_{J_0(\mathfrak{P})} - P_{\mathfrak{P}}$.
\end{lemma}
\begin{proof}
	Suppose that an $A$-representation $H$ satisfies $\tilde\rho_H(P_{J_0(\mathfrak{P})} - P_{\mathfrak{P}}) = 1_H$.
	It is equivalent to $\tilde\rho_H(P_{J_0(\mathfrak{P})}) = 1_H$ and $\tilde\rho_H(P_{\mathfrak{P}}) = 0$.
	The first equality implies that $J_0(\mathfrak{P})$ acts non-degenerately on $H$ while the second equality implies that $\mathfrak{P}\in \ker \rho_H$.

	Consequently, $J_0(\mathfrak{P})/\mathfrak{P}\cong \mathbb{K}(K_{\mathfrak{P}})$ acts non-degenerately on $H$.
	By the representation theory of the C*-algebra of compact operators, $H$ is isomorphic to a direct sum of copies of $K_{\mathfrak{P}}$ and $\tilde\rho_H(\ssA) = \rho_H(A)'' \cong \pi_{\mathfrak{P}}(A)'' = \tilde\pi_{\mathfrak{P}}(\ssA)$.
	Therefore, $P_{J_0(\mathfrak{P})} - P_{\mathfrak{P}}$ equals to the support projection of $\pi_{\mathfrak{P}}$.
\end{proof}

\begin{proposition}\label{p:point-sp-typeI}
	The atomic spectrum of a correspondence $X$ over a Type I C*-algebra $A$ has the form
	\[
		\sigma_a(X) = \{ \mathfrak{P} \in \operatorname{Prim}A\colon \varphi(\mathfrak{P})X \neq \varphi(J_0(\mathfrak{P}))X \}.
	\]
	If $A = C_0(Y)$ is abelian, then the formula above simplifies to
	\[
		\sigma_a(X) = \{ y\in Y\colon \varphi(C_0(Y \setminus\{y\}))X \neq X\}.
	\]
\end{proposition}
\begin{proof}

	By Lemma~\ref{l:ssPI}, we have $\varphi(\mathfrak{P})X \neq \varphi(J_0(\mathfrak{P}))X$ if and only if $\tilde\varphi(P_{\mathfrak{P}})\ssX \neq \tilde\varphi(P_{J_0(\mathfrak{P})})\ssX$.
	The latter inequality holds if and only if $\tilde\varphi(P_{\pi_{\mathfrak{P}}}) = \tilde\varphi(P_{J_0(\mathfrak{P})} - P_{\mathfrak{P}}) \neq 0$.
	Therefore, we need to prove that $\tilde\varphi(P_{\pi_{\mathfrak{P}}}) \neq 0$ if and only if $P_{\pi_{\mathfrak{P}}} \leq P_X$.

	Suppose that $P_{\pi_{\mathfrak{P}}} \leq P_X$.
	We have $\ssA P_{\pi_{\mathfrak{P}}} \subset \ssA P_X$.
	Therefore, $\ssA P_{\pi_{\mathfrak{P}}}$ acts on $\ssX$ faithfully and $\tilde\varphi(P_{\pi_{\mathfrak{P}}}) \neq 0$.

	On the other hand, if $\tilde\varphi(P_{\pi_{\mathfrak{P}}}) \neq 0$, then $0\neq P_XP_{\pi_{\mathfrak{P}}} \leq P_{\pi_{\mathfrak{P}}}$.
	Since $P_{\pi_{\mathfrak{P}}}$ is the support projection of an irreducible representation, it is a minimal central projection.
	Therefore, we get $P_XP_{\pi_{\mathfrak{P}}} = P_{\pi_{\mathfrak{P}}} \leq P_X$.

	The formula for abelian C*-algebras is a special case of the general one.
	The simplification comes from the fact that all primitive ideals are maximal and of the form $C_0(Y\setminus\{y\})$.
\end{proof}

\section{Topological quivers}\label{s:quivers}
Here we apply the results of previous sections to topological quivers.
In our exposition we follow the initial reference \cite{MuhlyTomforde} with the reversed role of the maps $s$ and $r$.
We assume all topological spaces to be second countable, locally compact, and Hausdorff.

Let $E^0, E^1$ be two topological spaces together with continuous maps $s,r\colon E^1 \rightrightarrows E^0$ called source and range maps.
Let $\lambda = (\lambda_x)_{x\in E^0}$ be the collection of measures on $E^1$ such that $\operatorname{supp}\lambda_x = s^{-1}(x)$ for all $x\in E^0$.
We additionally require that the function
\[
	x\mapsto \int_{s^{-1}(x)}\xi(t)d\lambda_x(t)
\]
is in $C_c(E^0)$ for all $\xi\in C_c(E^1)$.
The tuple $(E^0, E^1, r, s, \lambda)$ satisfying the above conditions is called a \emph{topological quiver}.

We define a $C_0(E^0)$-valued inner product on $C_c(E^1)$ by
\[
	\langle \xi, \eta\rangle(x) = \int_{s^{-1}(x)}\overline{\xi(t)}\eta(t)d\lambda_x(t)
\]
for all $\xi,\eta\in C_c(E^1)$.
We denote the completion of $C_c(E^1)$ with respect to this inner product by $X = X(E)$ which is a Hilbert $C_0(E^0)$-module.
The left action by $f\in C_0(E^1)$ on $X$ is given by $(\varphi(f)\xi)(t) = f(r(t))\xi(t)$ for all $\xi\in C_c(E^1)$, $t\in E^1$, and extended to all $X$ by continuity.

Ideals of $C_0(E^0)$ are in bijection with open subsets $U\subset E^0$ via the assignment $U\mapsto C_0(U)\subset C_0(E^0)$.
The open subset corresponding to Katsura's ideal $\cI_{X}$ is denoted by $E^0_{\mathrm{reg}}$.
The points of $E^0_{\mathrm{reg}}$ are called \emph{regular vertices}.
There is a topological characterization of regular vertices in \cite[Proposition 3.15]{MuhlyTomforde}.

As in the previous section, we identify $\widehat{C_0(E^0)}$ with $E^0$.
\begin{theorem}\label{t:spectrum-quiver}
	Let $E = (E^0, E^1, r, s, \lambda)$ be a topological quiver and $X=X(E)$ be the associated correspondence.
	The atomic spectrum $\sigma_a(X)$ consists of points $x\in E^0$ with $\lambda_y(r^{-1}(x)) \neq 0$ for some $y\in E^0$.
\end{theorem}
\begin{proof}
	We will use the description of the atomic spectrum from Proposition~\ref{p:point-sp-typeI}.
	Suppose that $\lambda_y(r^{-1}(x)) = 0$ for all $y\in E^0$.
	Let $\{e_i\}_{i\in \N}$ be an increasing approximate identity of $C_0(E^0\setminus\{x\})$.
	Then, for every $\xi,\eta\in C_c(E^1)$ and $y\in E^0$, we have
	\begin{multline*}
		\lim_{i\to \infty} \langle \varphi(e_i)\xi, \eta\rangle (y) = \int_{s^{-1}(y)} e_i(r(t))\overline{\xi(t)}\eta(t)d\lambda_y(t) \\= \int_{s^{-1}(y)} \overline{\xi(t)}\eta(t)d\lambda_y(t) = \langle \xi, \eta\rangle(y)
	\end{multline*}
	since $e_i\circ r \xrightarrow{\lambda_y\text{-a.e.}} 1$.
	Therefore, $\varphi(C_0(E^0\setminus\{x\}))X = X$ and $x\notin \sigma_a(X)$.

	Conversely, suppose that $\lambda_y(r^{-1}(x)) > 0$ for some $y\in E^0$.
	Then, we have $\varphi(C_0(E^0\setminus\{x\}))X\subset \overline{\{\xi\in C_0(E^0) \colon \xi|_{r^{-1}(x)} = 0\}} \subsetneq X$.
	We conclude that $x\in \sigma_a(X)$.
\end{proof}

A topological quiver is called a topological graph if the map $s$ is a local homeomorphism.
This implies that $s^{-1}(x) = \operatorname{supp} \lambda_x$ is discrete for all $x\in E^0$.
\begin{corollary}\label{cor:conjecture-quiver}
	A topological quiver violates the hyperrigidity conjecture if and only if $\lambda_y(r^{-1}(x))=0$ for all $x\in E^0\setminus E^0_{\mathrm{reg}}$ and $y\in E^0$ while $\lambda_y(r^{-1}(E^0\setminus E^0_{\mathrm{reg}})) \neq 0$ for some $y\in E^0$.
	In particular, correspondences of topological graphs satisfy the hyperrigidity conjecture.
\end{corollary}
\begin{proof}
	By definition, we have $\widehat{\cI}_{X(E)} = E^0_{\mathrm{reg}}$.
	By Theorem~\ref{t:choquet} and Theorem~\ref{t:spectrum-quiver}, we have $\partial_C(X) = \widehat{\cO}_{X}$ if and only if $\sigma_a(X) = \widehat{\cI}_X$ if and only if $\lambda_y(r^{-1}(x))=0$ for all $x\in E^0\setminus E^0_{\mathrm{reg}}$.
	On the other hand, we have $\varphi(\cI_X)X = X$ if and only if all the measures $\lambda_x$ are concentrated on $r^{-1}(E^0_{\mathrm{reg}})$.
	By Kim's Theorem (Corollary~\ref{cor:kim}), it follows that $X$ is hyperrigid if and only if $\lambda_y(r^{-1}(E^0\setminus E^0_{\mathrm{reg}})) = 0$ for all $y\in E^0$.

	When $E$ is a topological quiver, the measure $\lambda_x$ has the discrete support $s^{-1}(x)$ and thus it is a discrete measure for all $x\in X$.
	Assume that $\lambda_y(r^{-1}(x))=0$ for all $x\in E^0\setminus E^0_{\mathrm{reg}}$.
	We have $\lambda_y(r^{-1}(E^0\setminus E^0_{\mathrm{reg}})) = \sum_{x\in E^0\setminus E^0_{\mathrm{reg}}} \lambda_y(r^{-1}(x)) = 0$ for all $y\in E^0$.
	Therefore, the two conditions for the violation of the hyperrigidity conjecture cannot both hold for topological graphs.
\end{proof}

\begin{example}
	Let $Y$ be a connected topological space which is not a single point and $\mu$ be a probability measure on $Y$ with full support.
	Then, we may define a topological quiver with $E^0 = Y\sqcup \{\bullet\}$, $E^1 = Y$, $s(Y) = \bullet$, $r = \operatorname{id}_Y$, and $\lambda_\bullet = \mu$.
	Since $s$ is not a homeomorphism on any open set, $E^0_{\mathrm{reg}} = \emptyset$.
	The atomic spectrum of $E$ is exactly the set of atoms of $\mu$.
	Therefore, if $\mu$ is atomless, then $E$ violates the hyperrigidity conjecture.
	One may check manually that the representation induced by $L^2(Y, \mu)$ is not maximal.
\end{example}

Let $Y$ be a topological space.
We identify $C_0(Y)^*$ with the set of bounded complex Radon measures on $Y$.
Let $B(Y)$ be the $*$-algebra of bounded Borel functions on $Y$.
Then, integration defines an injective $*$-homomorphism $B(Y) \hookrightarrow C_0(Y)^{**}$ and we identify $B(Y)$ with its image.
In particular, if $S\subset Y$ is a Borel subset, then the characteristic function $\mathbf{1}_S$ is a projection in the W*-algebra $C_0(Y)^{**}$.

For a representation $\rho_H\colon C_0(Y) \to \cB(H)$, we say that the projection $\tilde\rho_H(\mathbf{1}_S)$ is the \emph{spectral projection} corresponding to $S$.
If $Y = \sigma(T)$ for a normal operator $T\in \cB(H)$ and the representation $\rho_H$ is the tautological one, then the notion of spectral projection coincides with the classical one for normal operators.

If $Z$ is another topological space, then by Gelfand--Naimark duality, a $*$-homomorphism $\psi\colon C_0(Y)\to \mathcal{M}(C_0(Z)) = C_b(Z)$ is equivalent to a continuous map $\check{\psi}\colon Z\to Y$.
The dual of $\psi$ restricted to ${C_0(Z)}^*$ acts as the pushforward of measures $\check{\psi}_*\colon {C_0(Z)}^* \to {C_0(Y)}^*$.
Finally, the dual of $\check{\psi}_*$ is the normal homomorphism $\tilde\psi\colon {C_0(Y)}^{**}\to {C_0(Z)}^{**}$.
If $f \in B(Y)\subset {C_0(Y)}^{**}$, then $\tilde\psi(f) = f\circ\psi \in B(Z)$.
\begin{lemma}\label{l:supp-im}
	The support projection of $\psi\colon C_0(Y) \to \mathcal{M}(C_0(Z))$ is the characteristic function $\mathbf{1}_{\check{\psi}(Z)}\in B(Y)\subset C_0(Y)^{**}$ of the Borel set $\check{\psi}(Z)\subset Y$.
\end{lemma}
\begin{proof}
	Denote $C = \check{\psi}(Z)$.
	We have $\tilde \psi(\mathbf{1}_{C}) = \psi\circ \mathbf{1}_{C} = 1$ so that $\mathbf{1}_{C} \geq P_{\psi}$.
	Suppose that $\mathbf{1}_{C} > P_{\psi}$.
	Hence, there exists a probability measure $\mu\in {C(Y)}^*$ such that $\mathbf{1}_{C}(\mu) = \mu(C) > P_{\psi}(\mu)$.
	By replacing $\mu$ with $\frac{1}{\mu(C)}\mu|_{C}$ we may assume that $\mu(C) = 1$ and $\mu(Y\setminus C)=1$.
	Since we assume $Y$ to be second countable LCH and thus separable and metrizable, we may apply \cite[Lemma 2.2]{VaradPullback} to find a probability measure $\nu$ on $Z$ such that $\mu = \check{\psi}_*\nu$.
	We arrive at a contradiction by noting that
	\[
		1 > P_\psi(\mu) = P_\psi(\check \psi_*\nu) = \tilde \psi(P_\psi)(\nu) = 1.
	\]
	We conclude that $\mathbf{1}_{C} = P_{\psi}$.
\end{proof}

\begin{theorem}\label{t:topological-graphs}
	A representation $(H, t)$ of the correspondence $X=X(E)$ of a topological graph $E = (E^0, E^1, r, s, \lambda)$ is fully Cuntz-Pimsner covariant if and only if $\overline{t(X)H} = H_{r(E^1)}$, where $H_{r(E^1)}$ is the range of the spectral projection on $H$ corresponding to the Borel subset $r(E^1)\subset E^0$.
\end{theorem}
\begin{proof}
	Since $s$ is a local homeomorphism, from \cite[Theorem 3.11]{MuhlyTomforde} it follows that pointwise multiplication defines a natural embedding $\iota\colon C_0(E^1) \hookrightarrow \cK(X(E))$.
	Consequently, the induced map $\tilde\iota\colon {C_0(E^1)}^{**} \to \cL({X(E)}^{**})$ is injective.
	The left action is given by the composition \[\varphi\colon C_0(E^0) \xrightarrow{\hat{r}} C_b(E^1) = \mathcal{M}(C_0(E^1)) \xrightarrow{\iota} \mathcal{M}(\cK(X)) = \cL(X(E)).\]
	Therefore, the map $\tilde \varphi$ is given by the composition ${C_0(E^0)}^{**} \xrightarrow{\tilde r} {C_0(E^1)}^{**} \hookrightarrow \cL({X(E)}^{**})$ and thus the kernel of $\tilde\varphi$ equals the kernel of $\tilde r$.
	Finally, the support projection $P_X$ of $\varphi$ equals the support projection of $r$ which is $\mathbf{1}_{r(E^1)}$ by Lemma~\ref{l:supp-im}.
	We conclude that the representation is fully Cuntz-Pimsner covariant if and only if $t(X(E))H = H_{r(E^1)} = \tilde\rho(\mathbf{1}_{r(E^1)})H$.
\end{proof}

The above theorem generalizes the analogous result for discrete graphs by Dor-On and Salomon which was explained earlier in Example~\ref{ex:dor-on-salomon}.

\printbibliography
\end{document}
